\documentclass[11pt]{amsart}
\usepackage{amsmath,amssymb,amsfonts,amsthm,amscd,textcomp,times,booktabs}
\usepackage[dvipdfmx]{graphicx}
\usepackage{braket}
\usepackage{enumerate}
\usepackage{color,float}
\usepackage{bm}
\usepackage{multirow}
\usepackage{arydshln}
\usepackage[all]{xy}
\usepackage[normalem]{ulem}
\usepackage{mathtools}
\newtheorem{theorem}{Theorem}

\newtheorem{proposition}[theorem]{Proposition}
\newtheorem{lemma}[theorem]{Lemma}
\newtheorem{corollary}[theorem]{Corollary}

\theoremstyle{definition}
\newtheorem{definition}[theorem]{Definition}
\newtheorem{remark}[theorem]{Remark}
\newtheorem{example}[theorem]{Example}

\numberwithin{equation}{section}
\numberwithin{theorem}{section}

\usepackage{scalerel}
\usepackage{stackengine}
\stackMath
\newsavebox\tmpbox

\makeatletter
\newcommand*{\rom}[1]{\ensuremath{\mathrm{\expandafter\@slowromancap\romannumeral #1@}}}
\makeatother
\newcommand{\norm}[1]{\ensuremath{\left| #1 \right|}}
\newcommand{\even}{\ensuremath{\mathrm{ev}}}
\newcommand{\odd}{\ensuremath{\mathrm{od}}}
\newcommand{\cop}{\ensuremath{\mathrm{cp}}}
\newcommand{\divides}[2]{\ensuremath{{#1}\, |\,{#2}}}
\newcommand{\bdm}[1]{\ensuremath{\mbox{\boldmath $#1$}}}
\newcommand{\Mod}[1]{\ensuremath{\ (\mathrm{mod}\ #1)}}
\newcommand{\floor}[1]{\ensuremath{\lfloor #1\rfloor}}
\begin{document}
\title[Polynomial sequences and the minimal polynomial of $2\cos (2\pi /n)$]
{Polynomial sequences related to Chebyshev polynomials\\
and the minimal polynomial of $\bdm{2\cos (2\pi /n)}$}
\author{Mamoru Doi}
\address{11-9-302 Yumoto-cho, Takarazuka, Hyogo 665-0003, Japan}
\email{doi.mamoru@gmail.com}
\maketitle
\noindent{\bfseries Abstract.}~
In this paper we consider the minimal polynomial $\psi_n(x)$ of $2\cos (2\pi /n)$.
We introduce some polynomial sequences with the same recurrence relation as the rescaled Chebyshev polynomials
$t_n(x)=2\, T_n(x/2)$ of the first kind,
which turn out to be related to those of various kinds, all coming from those of the second kind.
We see that $t_n(x)\pm 2=2(T_n(x/2)\pm 1)$ are divisible by the square of either of these polynomials.
Then by appropriately removing unnecessary factors from these polynomials,
we can easily calculate $\psi_n(x)$ without recursion, which improves Barnes' result in $1977$.
As an appendix, we give a compact table of the minimal polynomials $\psi_n(x)$ of $2\cos (2\pi /n)$ for $n\leqslant 120$.
\section{Introduction}\label{sec:intro}

The minimal polynomials $\Psi_n(x)$ and $\psi_n(x)$ of $\cos (2\pi /n)$ and $2\cos (2\pi /n)$ related by
\begin{equation}\label{eq:psi_Psi}
\psi_n(x)=2^{\deg\psi_n(x)}\Psi_n(x/2)
\end{equation}
where $\deg\psi_n(x)$ is given in $\eqref{eq:psi}$ below,
have long been studied in relation to cyclotomic polynomials and Chebyshev polynomials.
It is well-known that $2\cos (2\pi /n)$ is an algebraic integer for each natural number $n$,
so that $\psi_n(x)$ is a monic polynomial with integer coefficients.
This follows from the classical result of Lehmer \cite{Leh33} that $\psi_n(x)$ for $n>2$
is related to the $n$-th cyclotomic polynomial $\Phi_n(x)$ by the formula
\begin{equation*}
\psi_n\left( x+x^{-1}\right) =x^{-\phi (n)/2}\Phi_n(x),\quad\text{where}\quad\Phi_n(x)
=\prod_{\substack{0<k<n,\\ \gcd(k,n)=1}}\left( x-e^{2k\pi i/n}\right)
\end{equation*}
and $\phi (n)$ denotes Euler's totient function.

In $1977$, Barnes \cite{Bar77} proved using cyclotomic polynomials some formulas
as will be stated in Theorem $\ref{thm:Barnes}$ below,
which express $\psi_n(x)$ in terms of polynomials related to the Chebyshev polynomials of the second, third and fourth kinds.
However, until recently this result seems to have been almost forgotten in spite of its usefulness.

Meanwhile, in $1993$ Watkins and Zeitlin \cite{WZ93} calculated the minimal polynomial $\Psi_n(x)$ of $\cos (2\pi /n)$
using the Chebyshev polynomials $T_n(x)$ of the first kind.
In addition to $\eqref{eq:psi_Psi}$, rescaling $T_n(x)$ to monic polynomials $t_n(x)$ with integer coefficients by
\begin{equation}\label{eq:t}
t_n(x)=2\, T_n(x/2),\quad\text{so that}\quad t_n(2\cos\theta )=2\cos n\theta ,
\end{equation}
their main results are rephrased in terms of $\psi_n(x)$ and $t_n(x)$, instead of $\Psi_n(x)$ and $T_n(x)$, as follows.
\begin{theorem}[\cite{WZ93}, p. $473$, Lemma]\label{thm:WZ_lem}
The minimal polynomial $\psi_n(x)$ of $2\cos (2\pi/n)$ is given by
\begin{equation}\label{eq:psi}
\psi_n(x)=\prod_{\substack{0<k<n/2,\\ \gcd(k,n)=1}}\left( x-2\cos\frac{2k}{n}\pi\right),\quad
\deg\psi_n(x)=\begin{dcases}1&\text{if }n=1,2,\\ \phi (n)/2&\text{if }n>2.\end{dcases}
\end{equation}
\end{theorem}
\begin{theorem}[\cite{WZ93}, p. $471$, Theorem]\label{thm:WZ}
We have
\begin{equation*}
\prod_{\divides{d}{n}}\psi_d(x)=\prod_{k=0}^s\left( x-2\cos\frac{2k}{n}\pi\right) =\begin{dcases}
t_{s+1}(x)-t_s(x)&\text{if }n=2s+1\text{ is odd},\\
t_{s+1}(x)-t_{s-1}(x)&\text{if }n=2s\text{ is even}.
\end{dcases}
\end{equation*}
\end{theorem}
Unlike Barnes' method, calculating $\psi_n(x)$ by Theorem $\ref{thm:WZ}$
requires recursive calculations of $\psi_d(x)$ for $d<n$ with $\divides{d}{n}$.
Recently there have been obtained many such formulas
for factoring various polynomials by some of the minimal polynomials $\psi_n(x)$ or $\Psi_n(x)$
(see e.g. G\"{u}rta\c{s} \cite{Gur17, Gur22}, Wolfram \cite{Wol22-1, Wol22-2} and K\'{e}ri \cite{Ker22}).

In this paper, we shall partially improve Barnes' result in $1977$ without using cyclotomic polynomials
to obtain simpler formulas to calculate the minimal polynomial $\psi_n(x)$ of $2\cos (2\pi /n)$,
or more generally, the minimal polynomial $\psi_{m/n}(x)$ of $2\cos (m\pi /n)$ for any irreducible fraction $m/n\in (0,1)$.
For this purpose, we introduce polynomial sequences $\{ c_n(x)\}$, $\{ p^\pm_n (x)\}$ and $\{ q^\pm_n (x)\}$ as follows,
which will turn out to be obtained by rescaling the Chebyshev polynomials of various kinds
(see Remark $\ref{rem:main}$, $(2)$ below).
Let us first define a polynomial sequence $\{ c_n(x)\}$ by
\begin{equation}\label{eq:c_rec}
\begin{gathered}
c_{-2}(x)=-1,\quad c_{-1}(x)=0,\quad c_0(x)=1,\quad c_1(x)=x,\quad\text{and}\\
c_n(x)=x\cdot c_{n-1}(x)-c_{n-2}(x)\quad\text{for }n\geqslant 2.
\end{gathered}
\end{equation}
As will be proved in Proposition $\ref{prop:expand_c}$, we can also expand $c_n(x)$ as
\begin{equation}\label{eq:expand_c_1}
c_n(x)=\sum_{k=0}^{\floor{n/2}}(-1)^k\binom{n-k}{k}\, x^{n-2k}\quad\text{for }n\geqslant 0.
\end{equation}
Then polynomial sequences $\{ p^\pm_n(x)\}$ and $\{ q^\pm_n (x)\}$ are defined by
\begin{equation}\label{eq:pq_c}
p^\pm_n (x)=c_n(x)\pm c_{n-1}(x)\quad\text{and}\quad q^\pm_n (x)=c_n(x)\pm c_{n-2}(x)\quad\text{for }n\geqslant 0,
\end{equation}
where and whereafter all double signs correspond.
All these polynomial sequences have the same recurrence relation
as the rescaled Chebyshev polynomial sequence $\{ t_n(x)\}$,
but different initial values except for $\{ q^-_n(x)\}=\{ t_n(x)\}$ (see the beginning of Section $\ref{sec:props}$).
As will be seen in Theorem $\ref{thm:factor_t}$,
$t_n(x)+2$ and $t_n(x)-2$ are respectively divisible by $p^-_s(x)^2$ and $p^+_s(x)^2$ if $n=2s+1$ is odd, 
and $q^-_s(x)^2$ and $q^+_s(x)^2/x^2=c_{s-1}(x)^2$ if $n=2s$ is even.

We also need the following definition.
\begin{definition}\label{def:Pi}
We define a set $\Pi_i(n)$ by
\begin{equation*}
\Pi_i(n)=\set{p_1\cdots p_i<n\, |\, p_1,\dots ,p_i\text{ are distinct \emph{odd} prime divisors of }n}.
\end{equation*}
In particular, $\Pi_1(n)$ is the set of odd prime divisors of $n$.
Also, $\Pi_i(n)$ is empty if $i>\#\Pi_1(n)$ or $n=p_1\cdots p_i$ for distinct primes $p_1,\dots ,p_i$.
\end{definition}
Then our main result is stated as follows.
\begin{theorem}\label{thm:main}
Suppose $n>2$.
Let $m/n\in (0,1)$ be an irreducible fraction.
Then the minimal polynomial $\psi_{m/n}(x)$ of $2\cos (m\pi /n)$ is given as follows.

\noindent$(\mathrm{i})$ If both $n$ and $m$ are odd, the we have
\begin{align}
\psi_{m/n}(x)=\psi_{1/n}(x)=\psi_{2n}(x)
&=\left. p^-_{\floor{n/2}}(x)\middle/\prod_{2<d<n,\, \divides{d}{n}}\psi_{2n/d}(x)\right.
\label{eq:p-_psi}\\
&=p^-_{\floor{n/2}}(x)\cdot\prod_{i=1}^{\#\Pi_1(n)}
\left(\prod_{d\in\Pi_i(n)}p^-_{\floor{n/2d}}(x)\right)^{(-1)^i}.
\label{eq:psi_p-}
\end{align}
In particular, if $n=p^\ell$ for an odd prime $p$, then we have
\begin{equation*}
\psi_{m/n}(x)=\psi_{1/n}(x)=\psi_{2n}(x)=\begin{dcases}
p^-_{\floor{n/2}}(x)&\text{if }\ell =1,\\
\left. p^-_{\floor{n/2}}(x)\middle/ p^-_{\floor{n/2p}}(x)\right.&\text{if }\ell >1.
\end{dcases}
\end{equation*}
$(\mathrm{ii})$ If $n$ is odd and $m$ is even, then we have
\begin{align}
\psi_{m/n}(x)=\psi_{2/n}(x)=\psi_n(x)
&=\left. p^+_{\floor{n/2}}(x)\middle/\prod_{2<d<n,\,\divides{d}{n}}\psi_{n/d}(x)\right.
\label{eq:p+_psi}\\
&=p^+_{\floor{n/2}}(x)\cdot\prod_{i=1}^{\#\Pi_1(n)}
\left(\prod_{d\in\Pi_i(n)}p^+_{\floor{n/2d}}(x)\right)^{(-1)^i}.
\label{eq:psi_p+}
\end{align}
In particular, if $n=p^\ell$ for an odd prime $p$, then we have
\begin{equation*}
\psi_{m/n}(x)=\psi_{2/n}(x)=\psi_n(x)=\begin{dcases}
p^+_{\floor{n/2}}(x)&\text{if }\ell =1,\\
\left. p^+_{\floor{n/2}}(x)\middle/ p^+_{\floor{n/2p}}(x)\right.&\text{if }\ell >1.
\end{dcases}
\end{equation*}
$(\mathrm{iii})$ If $n=2^j n'$ is even and $m$ is odd, with $j>0$ and odd $n'$, then we have
\begin{align}
\psi_{m/n}(x)=\psi_{1/n}(x)=\psi_{2n}(x)
&=\left. q^-_{n/2}(x)\middle/\prod_{\substack{2<d\leqslant n',\, \divides{d}{n'}}}\psi_{2n/d}(x)\right.
\label{eq:q-_psi}\\
&=q^-_{n/2}(x)\cdot\prod_{i=1}^{\#\Pi_1(n)}
\left(\prod_{d\in\Pi_i(n)}q^-_{n/2d}(x)\right)^{(-1)^i}.
\label{eq:psi_q-}
\end{align}
In particular, if $n=2^j p^\ell$ for an odd prime $p$, then we have
\begin{equation*}
\psi_{m/n}(x)=\psi_{1/n}(x)=\psi_{2n}(x)=\begin{dcases}
q^-_{n/2}(x)&\text{if }\ell=0,\\
\left. q^-_{n/2}(x)\middle/ q^-_{n/2p}(x)\right.&\text{if }\ell >0.
\end{dcases}
\end{equation*}
\end{theorem}
We shall give an elementary proof of the above theorem without cyclotomic polynomials in Section $\ref{sec:proof}$
using some identities among $\{ c_n(x)\}$, $\{ p^\pm (x)\}$ and $\{ q^\pm (x)\}$.
\begin{remark}\label{rem:main}
$(1)$ The polynomial sequences $\{ c_n(x)\}$, $\{ p^+_n(x)\}$ and $\{ p^-_n(x)\}$ were introduced by
Barnes \cite{Bar77} as $\{ S_n(x)\}$, $\{ f_n(x)\}$ and $\{ g_n(x)\}$ respectively.
Also, it must be noted in relation to Theorem $\ref{thm:main}$
that Barnes proved the following result using cyclotomic polynomials.
\begin{theorem}[\cite{Bar77}, Theorem $11$]\label{thm:Barnes}
The minimal polynomial $\psi_n(x)$ of $2\cos (2\pi /n)$ is expressed as follows:

\noindent$\phantom{ii}(\mathrm{i})$ If $n=2s+1$ is odd, then
$\displaystyle\psi_n(x)=\prod_{d>1,\,\divides{d}{n}}f_{\floor{d/2}}(x)^{\mu (n/d)}$;

\noindent$\phantom{i}(\mathrm{ii})$ If $n=2s+1$ is odd, then
$\displaystyle\psi_{2n}(x)=\prod_{d>1,\,\divides{d}{n}}g_{\floor{d/2}}(x)^{\mu (n/d)}$; and

\noindent$(\mathrm{iii})$ If $n=2s$ is even, then
$\displaystyle\psi_{2n}(x)=\prod_{d>1,\,\divides{d}{n}}S_{d-1}(x)^{\mu (s/d)}$,

\noindent where $f_n(x)=p^+_n(x)$, $g_n(x)=p^-_n(x)$ and $S_n(x)=c_n(x)$,
and $\mu(n)$ denotes the M\"{o}bius function.
\end{theorem}
The expressions of $\psi_n(x)$ in Theorem $\ref{thm:Barnes}$, $(\mathrm{i})$ and $(\mathrm{ii})$
are respectively the same as $\eqref{eq:psi_p+}$ and $\eqref{eq:psi_p-}$ in Theorem $\ref{thm:main}$,
according to the definition of the M\"{o}bius function.
Meanwhile, the expression of
$\psi_{2n}(x)$ for $n=2s$ in Theorem $\ref{thm:Barnes}$, $(\mathrm{iii})$ is different from $\eqref{eq:psi_q-}$,
and needs more terms including $S_{n-1}(x)=c_{n-1}(x)$ (see also remark $(5)$ below).

\noindent $(2)$ According to Barnes \cite{Bar77} and K\'{e}ri \cite{Ker22},
$c_n (x)$, $p^-_n (x)$, $p^+_n (x)$ and $q^-_n(x)$ are respectively obtained by
rescaling $U_n(x)$, $V_n(x)$, $W_n(x)$ and $T_n(x)$
which are the Chebyshev polynomials of the second, third, fourth and first kinds, by
\begin{align*}
c_n(x)=U_n(x/2),\quad p^-_n(x)=V_n(x/2),\quad p^+_n(x)=W_n(x/2)\quad\text{and}\quad q^-_n(x)=2\, T_n(x/2),
\end{align*}
where $V_n(x)$ and $W_n(x)$ were introduced by Mason and Handscomb \cite{MH02}.
Thus $\eqref{eq:psi_p-}$, $\eqref{eq:psi_p+}$ and $\eqref{eq:psi_q-}$ are also regarded as
the factorization of $V_n(x/2)$, $W_n(x/2)$ and $2\, T_n(x/2)$ by the minimal polynomials
\begin{align*}
V_n(x/2)&=p^-_n(x)=\prod_{d<n,\,\divides{d}{2n+1}}\psi_{2(2n+1)/d}(x)=\prod_{d>1,\,\divides{d}{2n+1}}\psi_{2d}(x),\\
W_n(x/2)&=p^+_n(x)=\prod_{d<n,\,\divides{d}{2n+1}}\psi_{(2n+1)/d}(x)=\prod_{d>1,\,\divides{d}{2n+1}}\psi_d(x),\quad\text{and}\\
2\,T_n(x/2)&=t_n(x)=q^-_n(x)=\prod_{\divides{d}{n},\, d:\,\textrm{odd}}\psi_{4n/d}(x)
=\prod_{\divides{d}{n},\, n/d:\,\textrm{odd}}\psi_{4d}(x)
\end{align*}
(see e.g. \cite{Yam13}, Proposition $2.4$, and \cite{Ker22}, equations $(54)$ and $(55)$).
Hence as a corollary, we have the following irreducibility conditions for $T_n(x)$, $V_n(x)$ and $W_n(x)$.
\begin{corollary}\label{cor:irred_TVW}
$\phantom{i}(\mathrm{i})$ {\rm (\cite{RTW05}, Corollary }$3$, $(1)${\rm )}
The Chebyshev polynomial $T_n(x)$ of the first kind is irreducible if and only if $n$ is a power of $2$.

\noindent$(\mathrm{ii})$ The Chebyshev polynomials $V_n(x)$ and $W_n(x)$ of the third and fourth kinds
are irreducible if and only if $2n+1$ is a prime.
\end{corollary}

\noindent $(3)$ In the right-hand sides of $\eqref{eq:p-_psi}$ and $\eqref{eq:p+_psi}$,
we can replace $n/d$ in the subscripts of $\psi$ with $d$,
but cannot in that of $\eqref{eq:q-_psi}$.
Also note that the product in the right-hand side of $\eqref{eq:q-_psi}$ includes the case $d=n'$,
while those of $\eqref{eq:p-_psi}$ and $\eqref{eq:p+_psi}$ do not include the case $d=n$.

\noindent$(4)$ If $n$ is odd, then $\psi_{2n}(x)$ is equal to $\psi_n(-x)$ up to sign.
This is because $\psi_{2n}(x)$ is obtained by replacing all $p^+$'s with $p^-$'s
in the expression of $\psi_n(x)$, which satisfy $p^+_s(-x)=(-1)^sp^-_s(x)$ (see Corollary $\ref{cor:factor_pq}$).

\noindent$(5)$ Let $n=2^jp_1^{\ell_1}\cdots p_i^{\ell_i}$ be the prime factorization of $n>2$.
Also, we set $\nu =1$ if $\ell_1=\dots =\ell_i=1$ and $\nu =0$ otherwise.
Then we see that
\begin{equation*}
\psi_n(x)\text{ is expressed by}
\begin{dcases}
(2^i-\nu )\text{ terms of }\{ p^+_k(x)\}&\text{if }n\equiv 1\Mod{2},\text{ or }j=0,\\
(2^i-\nu )\text{ terms of }\{ p^-_k(x)\}&\text{if }n\equiv 2\Mod{4},\text{ or }j=1,\\
2^i\text{ terms of }\{ q^-_k(x)\}&\text{if }n\equiv 0\Mod{4},\text{ or }j>1.
\end{dcases}
\end{equation*}
which also implies Corollary $\ref{cor:irred_TVW}$.
Meanwhile, according to Theorem $\ref{thm:Barnes}$, $(\mathrm{iii})$ due to Barnes,
$\psi_n(x)$ for $n\equiv 0\Mod{4}$ is expressed by
$j(\ell_1+1)\cdots (\ell_i+1)-1$ terms of $\{ c_k(x)\}$.

\noindent$(6)$ When $n=p$ for a prime $p>2$, $\psi_n(x)$ was expanded by Surowski and McCombs in \cite{SMC03}, Theorem $3.1$,
and Beslin and de Angelis in \cite{BdA04}, p. $146$ (see also the comment after Corollary $2.2$ in \cite{LW11}).
This is also immediate from Theorem $\ref{thm:Barnes}$, $(\mathrm{i})$ due to Barnes.
Also, when $n=p^\ell$ for a prime $p>1$, $\psi_n(x)$ was expressed as a sum of Chebyshev polynomials of the first kind
by Lang in \cite{Lang11}, Proposition.
Our theorem simplifies and generalizes these results.
\end{remark}
\begin{example}
To see how Theorem $\ref{thm:main}$ works, let us calculate the minimal polynomial $\psi_{60}(x)$ of $2\cos (\pi /30)$.
Then we can use $\eqref{eq:psi_q-}$ in Theorem $\ref{thm:main}$, $(\mathrm{iii})$ for $n=30$, $m=1$.
We see that an odd divisor $d$ of $30$ with $2<d<30$ is either $3$, $5$ or $15=3\cdot 5$.
Thus from Definition $\ref{def:Pi}$ we have $\Pi_1(30)=\{ 3,5\}$, $\Pi_2(30)=\{ 15\}$ and $\Pi_i(30)=\emptyset$ for $i>2$.
Consequently, we can calculate $\psi_{60}(x)$ as
\begin{equation}\label{eq:psi_60}
\begin{aligned}
\psi_{60}(x)&=q^-_{30/2}(x)\cdot\frac{\prod_{d\in\Pi_2(30)}q^-_{30/2d}(x)}{\prod_{d\in\Pi_1(30)}q^-_{30/2d}(x)}
=\frac{q^-_{15}(x)\, q^-_1(x)}{q^-_5(x)\, q^-_3(x)}\\
&=\frac{(c_{15}(x)-c_{13}(x))\, (c_1(x)-c_{-1}(x))}{(c_5(x)-c_3(x))\, (c_3(x)-c_1(x))}\\
&=x^8-7x^6+14x^4-8x^2+1,
\end{aligned}
\end{equation}
where we used $\eqref{eq:pq_c}$ and
\begin{align*}
c_{-1}(x)&=0,\quad c_1(x)=x,\quad c_3(x)=x^3-2x,\quad c_5(x)=x^5-4x^3+3x,\\
c_{13}(x)&=x^{13}-12x^{11}+55x^9-120x^7+126x^5-56x^3+7x,\\
c_{15}(x)&=x^{15}-14x^{13}+78x^{11}-220x^9+330x^7-252x^5+84x^3-8x,
\end{align*}
which are obtained by either $\eqref{eq:c_rec}$ or $\eqref{eq:expand_c_1}$.
We can also calculate $q^-_n(x)$ directly by Proposition $\ref{prop:expand_q-}$.
This improves the expression of $\psi_{60}(x)$ by Theorem $\ref{thm:Barnes}$, $(\mathrm{iii})$ due to Barnes that
\begin{equation*}
\psi_{60}(x)=\frac{c_{29}(x)\, c_4(x)\, c_2(x)\,c_1(x)}{c_{14}(x)\, c_9(x)\,c_5(x)},
\end{equation*}
which has more terms and includes a term of higher degree than $\eqref{eq:psi_60}$.
Also, we can express $\psi_{60}(x)$ using $\eqref{eq:q-_psi}$ as
\begin{equation*}
\psi_{60}(x)=\frac{q^-_{15}(x)}{\psi_4(x)\,\psi_{12}(x)\,\psi_{20}(x)}
=\frac{t_{15}(x)}{\psi_4(x)\,\psi_{12}(x)\,\psi_{20}(x)},
\end{equation*}
which simplifies the expression using Theorem $\ref{thm:WZ}$ due to Watkins and Zeitlin that
\begin{equation*}
\psi_{60}(x)=\frac{t_{31}(x)-t_{29}(x)}
{\psi_1(x)\,\psi_2(x)\,\psi_3(x)\,\psi_4(x)\,\psi_5(x)\,\psi_6(x)\,
\psi_{10}(x)\,\psi_{12}(x)\,\psi_{15}(x)\,\psi_{20}(x)\,\psi_{30}(x)}.
\end{equation*}
\end{example}
To illustrate the idea behind Theorem $\ref{thm:main}$, let us observe the above calculation of $\psi_{60}(x)$ more closely.
Let $\Sigma_\odd (n)=\{ 0<k<n:\,\text{odd}\}$ and $\Sigma_\cop (n)=\{ 0<k<n:\,\text{coprime to }n\}$.
Then we want to calculate
\begin{equation*}
\psi_{60}(x)=\prod_{k\in\Sigma_\cop (30)}\left( x-2\cos\frac{k}{30}\pi\right)
\end{equation*}
(see also Proposition $\ref{prop:Sigma_cp}$, $(\mathrm{ii})$).
Meanwhile, as will turn out in Corollary $\ref{cor:factor_pq}$, if $n=2s$ is even, then we have
\begin{equation}\label{eq:factor_q-}
q^-_s(x)=\prod_{k=1}^{s}\left( x-2\cos\frac{2k-1}{2s}\pi\right)
=\prod_{k\in\Sigma_\odd (n)}\left( x-2\cos\frac{k}{n}\pi\right) .
\end{equation}
Thus taking $n=30$ in $\eqref{eq:factor_q-}$, so that $s=15$, we see
from $\Sigma_\cop (30)\subset\Sigma_\odd (30)$ that $\psi_{60}(x)$ divides $q^-_{15}(x)$.
More specifically, we have
\begin{align*}
\Sigma_\cop (30)&=\Sigma_\odd (30)\setminus\{ 3,5,9,15,21,25,27\}\\
&=\Sigma_\odd (30)\setminus 3\,\Sigma_\odd (30/3)\cup 5\,\Sigma_\odd (30/5),
\end{align*}
where $3$ and $5$ appear as distinct odd prime divisors of $n=30$.
Consequently, noting that $3\,\Sigma_\odd (30/3)\cap 5\,\Sigma_\odd (30/5)=\{ 15\} =15\,\Sigma_\odd (30/15)$
and using $\eqref{eq:factor_q-}$ again,
we obtain the desired expression $\eqref{eq:psi_60}$ of $\psi_{60}(x)$.

The minimal polynomials $\psi_n(x)$ for $n\leqslant 120$ are given in Appendix A
in terms of $\{ p^\pm_k(x)\}$ and $\{ q^-_k(x)\}$, all of which come from $\{ c_k(x)\}$.
\section{Properties of the polynomials $c_n(x)$, $p^\pm_n(x)$ and $q^\pm_n(x)$}\label{sec:props}
Recall that the Chebyshev polynomial sequence $\{ T_n(x)\}$ satisfies
\begin{equation*}
T_0(x)=1,\quad T_1(x)=x,\quad\text{and}\quad T_n(x)=2x\cdot T_{n-1}(x)-T_{n-2}(x)\quad\text{for }n\geqslant 2.
\end{equation*}
Accordingly, the rescaled Chebyshev polynomial sequence $\{ t_n(x)\}$ defined in $\eqref{eq:t}$ satisfies
\begin{equation}\label{eq:t_rec}
t_0(x)=2,\quad t_1(x)=x,\quad\text{and}\quad t_n(x)=x\cdot t_{n-1}(x)-t_{n-2}(x)\quad\text{for }n\geqslant 2,
\end{equation}
so that each $t_n(x)$ is a monic polynomial with integer coefficients.
We can also express $t_n(x)$ as
\begin{equation}\label{eq:t_lambda}
t_n(x)=\lambda_+^n+\lambda_-^n,
\end{equation}
where $\lambda_\pm$ are two solutions of the characteristic equation $\lambda^2-x\lambda+1=0$ of
the recurrence relation for $\{ t_n(x)\}$, given by
\begin{equation}\label{eq:lambda}
\lambda_\pm =\frac{x\pm\sqrt{x^2-4}}{2},\quad\text{which satisfy}\quad
\lambda_++\lambda_-=x\quad\text{and}\quad\lambda_+\lambda_-=1
\end{equation}
(see \cite{Riv90}, p. $5$, Exercise $1.1.1$).

Since we see from the definition $\eqref{eq:c_rec}$ of $\{ c_n(x)\}$
that the recurrence relation for $\{ c_n(x)\}$ is the same as $\{ t_n(x)\}$,
we can express $c_n(x)$ in terms of $\lambda_\pm$ as
\begin{equation}\label{eq:c_lambda}
c_n(x)=\frac{\lambda_+^{n+1}-\lambda_-^{n+1}}{\sqrt{x^2-4}}.
\end{equation}
Also, $\{ p^\pm_n(x)\}$ and $\{ q^\pm_n(x)\}$ satisfy the same recurrence relation as $\{ t_n(x)\}$
due to their definition $\eqref{eq:pq_c}$.
Thus $\{ p^\pm_n(x)\}$ and $\{ q^\pm_n(x)\}$ can be alternatively defined by
\begin{alignat*}{3}
p^\pm_0 (x)&=1,&p^\pm_1 (x)&=x\pm 1,\quad&p^\pm_n (x)&=x\cdot p^\pm_{n-1} (x)-p^\pm_{n-2}(x),\\
q^\pm_0(x)&=1\mp 1,\quad&q^\pm_1(x) &=x,&q^\pm_n (x)&=x\cdot q^\pm_{n-1} (x)-q^\pm_{n-2} (x).
\end{alignat*}
In particular, we see that
\begin{equation*}
q^-_n(x)=t_n(x)\quad\text{and}\quad q^+_n(x)=x\cdot c_{n-1}(x).
\end{equation*}
\begin{proposition}
We have
\begin{equation}\label{eq:prod_c}
(x^2-4)\, c_m(x)\, c_n(x)=t_{m+n+2}(x)-t_{\norm{m-n}}(x).
\end{equation}
\end{proposition}
\begin{proof}
This is straightforward from $\eqref{eq:t_lambda}$, $\eqref{eq:lambda}$ and $\eqref{eq:c_lambda}$.
\end{proof}
\begin{theorem}\label{thm:factor_t}
If $n=2s+1$, then we have
\begin{equation*}
t_n(x)\pm 2=(x\pm 2)\, p^\mp_s(x)^2.
\end{equation*}
Also, if $n=2s$, then we have
\begin{align*}
t_n(x)+2&=q^-_s(x)^2=t_s(x)^2,\quad\text{and}\\
t_n(x)-2&=(x^2-4)\, q^+_s(x)^2/x^2=(x^2-4)\, c_{s-1}(x)^2.
\end{align*}
\end{theorem}
\begin{proof}
If $n=2s+1$, then using $\eqref{eq:pq_c}$ and $\eqref{eq:prod_c}$ we have
\begin{align*}
(x^2-4)\, p^\pm_s (x)^2&=(x^2-4)\, (c_s(x)\pm c_{s-1}(x))^2\\
&=t_{2s+2}(x)+t_{2s}(x)-4\pm 2(t_{2s+1}(x)-x)\\
&=(x\pm 2)\, (t_{2s+1}(x)\mp 2),
\end{align*}
where we used $\eqref{eq:t_rec}$ for the last equality.
Similarly, if $n=2s$, then we have
\begin{align*}
(x^2-4)\, q^\pm_s (x)^2&=(x^2-4)\, (c_s(x)\pm c_{s-2}(x))^2\\
&=t_{2s+2}(x)+t_{2s-2}(x)-4\pm 2\{ t_{2s}(x)-(x^2-2)\}\\
&=(x^2-2\pm 2)\, (t_{2s}(x)\mp 2),
\end{align*}
where we used for the last equality
\begin{equation*}
t_{n+2}(x)-(x^2-2)\, t_n(x)+t_{n-2}(x)=0,
\end{equation*}
which is easily derived from $\eqref{eq:t_rec}$.
\end{proof}
\begin{corollary}\label{cor:factor_pq}
For $s>0$, we have
\begin{align*}
p^-_s(x)&=\prod_{k=1}^s \left(x-2\cos\frac{2k-1}{2s+1}\pi\right) ,\\
p^+_s(x)&=\prod_{k=1}^s \left(x-2\cos\frac{2k}{2s+1}\pi\right) =(-1)^sp^-_s(-x),\\
q^-_s(x)&=t_s(x)=\prod_{k=1}^s \left(x-2\cos\frac{2k-1}{2s}\pi\right) ,\quad\text{and}\quad\\
\frac{q^+_s(x)}{x}&=c_{s-1}(x)=\prod_{k=1}^{s-1}\left(x-2\cos\frac{k}{s}\pi\right) =
\begin{dcases}
p^-_{s'}(x)\, p^+_{s'}(x)&\text{if }s=2s'+1\text{ is odd},\\
q^-_{s'}(x)\, c_{s'-1}(x)&\text{if }s=2s'\text{ is even}.
\end{dcases}
\end{align*}
\end{corollary}
\begin{proof}
This is immediate from Theorem $\ref{thm:factor_t}$ because the roots of $t_n(x)\pm 2$ consist of $x=2\cos\theta$
corresponding to two values of $\theta\in [0,2\pi )$ with $\cos (n\theta )\pm 1=0$,
so that we can take $\theta\in (0,\pi )$.
\end{proof}
We remark that Lee and Wong \cite{LW11} studied some combinatorial properties of
polynomials $A_n(x)$ defined by
\begin{equation*}
A_n(x)=2^n\prod_{k=1}^n\left( x-\cos\frac{2k}{2n+1}\pi\right) ,
\end{equation*}
which are the same as $p^+_n(2x)=W_n(x)$ due to Corollary $\ref{cor:factor_pq}$.

We end this section by proving the expansion of $c_n(x)$ given by $\eqref{eq:expand_c_1}$
together with a useful expansion of $q^-_n(x)$.
\begin{proposition}\label{prop:expand_c}
We can expand $c_n(x)$ as
\begin{equation}\label{eq:expand_c_2}
c_n(x)=\sum_{k=0}^{\floor{n/2}}(-1)^k\binom{n-k}{k}\, x^{n-2k}\quad\text{for }n\geqslant 0.
\end{equation}
\end{proposition}
\begin{proof}
Since $\eqref{eq:expand_c_2}$ gives $c_0(x)=1$ and $c_1(x)=x$,
it remains to prove that $\{ c_n(x)\}$ given by $\eqref{eq:expand_c_2}$
satisfies the recurrence relation $c_n(x)+c_{n-2}(x)=x\cdot c_{n-1}(x)$
as in $\eqref{eq:c_rec}$ for $n\geqslant 2$.
We calculate $c_n(x)+c_{n-2}(x)$ and $x\cdot c_{n-1}(x)$ as
\begin{align}
c_n(x)+c_{n-2}(x)&=x^n-\sum_{k=0}^{\floor{n/2}-1}(-1)^k\left\{\binom{n-k-1}{k+1}-\binom{n-k-2}{k}\right\}x^{n-2k-2}
\quad\text{and}\label{eq:c_n+c_n-2}\\
x\cdot c_{n-1}(x)&=x^n-\sum_{k=0}^{\floor{(n-1)/2}-1}(-1)^k\binom{n-k-2}{k+1}\, x^{n-2k-2}.\label{eq:xc_n-1}
\end{align}
If $n$ is odd, then $\floor{n/2}=\floor{(n-1)/2}$, so that $\eqref{eq:c_n+c_n-2}$ and $\eqref{eq:xc_n-1}$ are equal
because of the recurrence relation
\begin{equation}\label{eq:binom_rec}
\binom{n-k-1}{k+1}=\binom{n-k-2}{k}+\binom{n-k-2}{k+1}.
\end{equation}
If $n=2s$ is even, then due to $\eqref{eq:binom_rec}$,
$\eqref{eq:c_n+c_n-2}$ differs from $\eqref{eq:xc_n-1}$ by the constant term corresponding to $k=s-1$,
which is calculated as
\begin{equation*}
(-1)^s\left\{\binom{s}{s}-\binom{s-1}{s-1}\right\}=0,
\end{equation*}
so that $\eqref{eq:c_n+c_n-2}$ is equal to $\eqref{eq:xc_n-1}$.
This proves that $\{ c_n(x)\}$ given by $\eqref{eq:expand_c_2}$ satisfies the recurrence relation in $\eqref{eq:c_rec}$.
\end{proof}
\begin{proposition}\label{prop:expand_q-}
We can expand $q^-_n(x)=t_n(x)=2\, T_n(x/2)$ as
\begin{equation}
\begin{aligned}\label{eq:expand_q-}
q^-_n(x)&=x^n-\sum_{k=0}^{\floor{n/2}-1}(-1)^k\frac{n}{k+1}\binom{n-k-2}{k}\, x^{n-2k-2}\\
&=\sum_{k=0}^{\floor{n/2}}(-1)^k\,\frac{n\cdot (n-k-1)!}{k!\, (n-2k)!}\,x^{n-2k}\quad\text{for }n\geqslant 1.
\end{aligned}
\end{equation}
\end{proposition}
\begin{proof}
For $n=1$, $\eqref{eq:expand_q-}$ gives $q^-_1(x)=x$ correctly.
For $n\geqslant 2$, similarly to $\eqref{eq:c_n+c_n-2}$ we have
\begin{equation*}
q^-_n(x)=x^n-\sum_{k=0}^{\floor{n/2}-1}(-1)^k\left\{\binom{n-k}{k}+\binom{n-k-1}{k+1}\right\}x^{n-2k-2},
\end{equation*}
which immediately leads to $\eqref{eq:expand_q-}$.
\end{proof}
\section{Proof of the main theorem}\label{sec:proof}
Now suppose $n>2$ and
define $\Sigma (n)$, $\Sigma_\odd (n)$, and $\Sigma_\even (n)$, $\Sigma_\cop (n)$ and $\Sigma_\cop^{1/2}(n)$ by
\begin{align*}
\Sigma (n)&=\set{k\,|\, 0<k<n} ,\quad
\Sigma_\odd (n)=\{ 2k-1\in\Sigma (n)\},\quad\Sigma_\even (n)=\{ 2k\in\Sigma (n)\} ,\\
\Sigma_\cop (n)&=\set{k\in\Sigma (n)\,|\, k\text{ is coprime to }n}\quad\text{and}\quad
\Sigma_\cop^{1/2}(n)=\Sigma_\cop (n)\cap \Sigma (n/2),
\end{align*}
where we set for consistency $\Sigma (n/2)=\{ 1,2,\dots ,s\}$ if $n=2s+1$ is odd.
In particular, we have
\begin{align}
&2\,\Sigma (n/2)=\Sigma_\even (n)\quad\text{for all }n>2,\quad\text{so that}\nonumber\\
&\left.\Sigma_\even (n)\middle\backslash 2\,\Sigma^{1/2}_\cop (n)\right.
=2\, \left(\Sigma (n/2)\middle\backslash\Sigma^{1/2}_\cop (n)\right) .
\label{eq:Sigma_ev-2Sigma_cp}
\end{align}
Also, we define $\pi^\Sigma_s(n,x)$ for a subset $\Sigma$ of $\Sigma (n)$ with $\#\Sigma =s$ by
\begin{equation*}
\pi^\Sigma_s(n,x)=\prod_{m\in\Sigma}\left( x-2\cos\frac{m}{n}\pi\right) ,
\end{equation*}
where we set $\pi^\Sigma_0(n,x)=1$ if $\Sigma$ is empty.
Then it follows from Theorem $\ref{thm:WZ_lem}$ that $\psi_n(x)$ is written as
\begin{equation}\label{eq:psi_pi}
\psi_n (x)=\pi^{2\,\Sigma^{1/2}_\cop (n)}_{\phi (n)/2}(n,x).
\end{equation}
Also, we see from Corollary $\ref{cor:factor_pq}$ that if $n=2s+1$ is odd, then $p^\pm_s(x)$ are written as
\begin{equation}\label{eq:p_pi}
p^-_s(x)=\pi^{\Sigma_\odd (n)}_s(n,x)\quad\text{and}\quad p^+_s(x)=\pi^{\Sigma_\even (n)}_s(n,x),
\end{equation}
and if $n=2s$ is even, then $q^\pm_s(x)$ are written as
\begin{equation}\label{eq:q_pi}
q^-_s(x)=\pi^{\Sigma_\odd (n)}_s(n,x)\quad\text{and}\quad q^+_s(x)/x=c_{s-1}(x)=\pi^{\Sigma (n)}_{s-1}(n,x).
\end{equation}
The following two lemmas are immediate.
\begin{lemma}\label{lem:pi_Sigma}
$(\mathrm{i})$ If $\Sigma\subset\Sigma'\subset\Sigma (n)$,
then $\pi^\Sigma_s (n,x)$ divides $\pi^{\Sigma'}_{s'}(n,x)$, where $\#\Sigma =s$ and $\#\Sigma'=s'$.

\noindent$(\mathrm{ii})$ If $d$ divides $n$ and $\Sigma\subset\Sigma (n/d)$,
then we have $\pi^{d\,\Sigma}_s(n,x)=\pi^\Sigma_s(n/d,x)$, where $\#\Sigma =s$.
\end{lemma}
\begin{lemma}\label{lem:dSigma}
Let $d$ be a divisor of $n$.

\noindent$\phantom{\mathrm{i}}(\mathrm{i})$ If $n$ is odd, then
$d\,\Sigma_\odd (n/d)\subset\Sigma_\odd (n)$ and $d\,\Sigma_\even (n/d)\subset\Sigma_\even (n)$ always hold.

\noindent$(\mathrm{ii})$ If $n$ is even, then $d\,\Sigma_\odd (n/d)\subset\Sigma_\odd (n)$ holds if and only if $d$ is odd.
\end{lemma}
\begin{proposition}For $s,s'>0$, we have the following:

\noindent$\phantom{i}(\mathrm{i})$ $p^\pm_s(x)$ divides $p^\pm_{s'}(x)$ if and only if $2s+1$ divides $2s'+1$; and

\noindent$(\mathrm{ii})$ $q^-_s(x)$ divides $q^-_{s'}(x)$ if and only if $s$ divides $s'$ and $s'/s$ is odd.
\end{proposition}
\begin{proof}
If $p^-_s(x)$ divides $p^-_{s'}(x)$, then $1/(2s+1)=d/(2s'+1)$ for some odd $d$,
so that $2s+1$ divides $2s'+1$.
Conversely, if $2s+1$ divides $2s'+1$, then $d=(2s'+1)/(2s+1)$ is odd.
Thus it follows from Lemma $\ref{lem:dSigma}$, $(\mathrm{i})$ that $d\,\Sigma_\odd (2s+1)\subset\Sigma_\odd (2s'+1)$.
Then applying Lemma $\ref{lem:pi_Sigma}$ together with $\eqref{eq:p_pi}$ leads to the assertion.
The other cases are similar.
\end{proof}
\begin{proposition}\label{prop:Sigma_cp}
$(\mathrm{i})$ If $n$ is odd, then we have
\begin{equation}\label{eq:Sigma_cp_n=odd}
2\,\Sigma^{1/2}_\cop (n)=\left.\Sigma_\even (n)\middle\backslash\bigcup_{p\in\Pi_1(n)}p\,\Sigma_\even (n/p)\right. .
\end{equation}
$(\mathrm{ii})$ If $n$ is even, then we have
\begin{equation}\label{eq:Sigma_cp_n=even}
\Sigma^{1/2}_\cop (n)
=\left.\Sigma_\odd (n/2)\middle\backslash\bigcup_{p\in\Pi_1(n)}p\,\Sigma_\odd (n/2p)\right. .
\end{equation}
Moreover, if $n/2$ is also even, then we have $\Sigma^{1/2}_\cop (n)=\Sigma_\cop (n/2)$.
\end{proposition}
\begin{proof}
$(\mathrm{i})$ If $n$ is odd, then noting that all prime divisors of $n$ are also odd, we see easily that
\begin{equation}\label{eq:cpl_Sigma_cp_n=odd}
\left.\Sigma (n/2)\middle\backslash\Sigma^{1/2}_\cop (n)\right.
=\bigcup_{p\in\Pi_1(n)}p\,\Sigma (n/2p).
\end{equation}
Taking the complement of $\Sigma (n/2)$ on both sides of $\eqref{eq:cpl_Sigma_cp_n=odd}$, multiplying by $2$,
and using $\eqref{eq:Sigma_ev-2Sigma_cp}$ leads to $\eqref{eq:Sigma_cp_n=odd}$.

\noindent$(\mathrm{ii})$ If $n$ is even, then we have $2k\notin\Sigma^{1/2}_\cop (n)$ for all $k$,
so that $\Sigma^{1/2}_\cop (n)\subset\Sigma_\odd (n/2)$.
Thus $k\in\Sigma_\odd (n/2)$ satisfies $k\notin\Sigma^{1/2}_\cop (n)$ if and only if $k=p\ell$ for some $p\in\Pi_1(n)$
and odd $\ell <n/2p$.
Hence we have
\begin{equation}\label{eq:cpl_Sigma_cp_n=even}
\left.\Sigma_\odd (n/2)\middle\backslash\Sigma^{1/2}_\cop (n)\right.
=\bigcup_{p\in\Pi_1(n)}p\,\Sigma_\odd (n/2p).
\end{equation}
Taking the complement of $\Sigma_\odd (n/2)$ on both sides of $\eqref{eq:cpl_Sigma_cp_n=even}$
leads to $\eqref{eq:Sigma_cp_n=even}$.
The last assertion is immediate because if $n/2$ is even,
then the prime divisors of $n/2$ are the same as those of $n$, including $2$.
\end{proof}
\begin{proof}[Proof of Theorem $\ref{thm:main}$.]
$(\mathrm{i})$ Suppose $n=2s+1$.
We shall divide the proof into the following steps.

\noindent$(\mathrm{a})$ \emph{Proof of the first expression.}

\noindent$(\mathrm{a}.1)$ \emph{Rewriting the first expression.}
As for the the last equality of $\eqref{eq:p-_psi}$, due to Corollary $\ref{cor:factor_pq}$, it suffices to prove
\begin{equation}\label{eq:p-_psi_2}
\prod_{k=1}^s\left(x-2\cos\frac{2k-1}{2s+1}\pi\right) =\prod_{d<n,\, \divides{d}{n}}\psi_{2n/d}(x).
\end{equation}
$(\mathrm{a}.2)$ \emph{Calculating the degree.}
According to the second equation of $\eqref{eq:psi}$,
the degree of the right-hand side of $\eqref{eq:p-_psi_2}$ is calculated as
\begin{equation*}
\sum_{d<n,\, \divides{d}{n}}\frac{\phi (2n/d)}{2}
=\sum_{\divides{d}{n}}\frac{\phi (2d)}{2}-\frac{\phi (2)}{2}
=\sum_{\divides{d}{n}}\frac{\phi (d)}{2}-\frac{1}{2}=\frac{n-1}{2}=s,
\end{equation*}
which is equal to that of the left-hand side,
where we used $\phi (2d)=\phi (d)$ because all divisors of $n$ are odd.

\noindent$(\mathrm{a}.3)$ \emph{Inclusion of the factors.}
For any $0<k<s$, let $d=\gcd (2k-1,n)<n$.
Then $(2k-1)/d$ is coprime to $2n/d$, so that
\begin{equation*}
x-2\cos\frac{2k-1}{n}\pi =x-2\cos\frac{2(2k-1)/d}{2n/d}\pi
\end{equation*}
is included in $\psi_{2n/d}(x)$ as a factor.
Hence $\eqref{eq:p-_psi_2}$, both sides of which are monic, is proved due to $(\mathrm{a}.1)$--$(\mathrm{a}.3)$.

\noindent$(\mathrm{b})$ \emph{Proof of the second expression.}
Next we shall prove the second expression $\eqref{eq:psi_p-}$ of $\psi_{2n}(x)$.

\noindent$(\mathrm{b}.1)$ \emph{Rewriting the minimal polynomial.}
We rewrite $\psi_{2n}(x)$ as
\begin{gather}
\begin{aligned}\label{eq:psi_2n_p-}
\psi_{2n}(x)&=\pi^{2\,\Sigma^{1/2}_\cop (2n)}_{\phi (2n)/2}(2n,x)=\pi^{\Sigma^{1/2}_\cop (2n)}_{\phi (n)/2}(n,x)\\
&=\pi^{\Sigma_\odd (n)\setminus\Sigma}_{\phi (n)/2}(n,x)
=\frac{\pi^{\Sigma_\odd (n)}_{\floor{n/2}}(n,x)}{\pi^\Sigma_{\#\Sigma}(n,x)}
=\frac{p^-_{\floor{n/2}}(x)}{\pi^\Sigma_{\#\Sigma}(n,x)},
\end{aligned}\\
\text{where}\quad\Sigma =\bigcup_{p\in\Pi_1(n)}p\,\Sigma_\odd (n/p),\notag
\end{gather}
and we used $\eqref{eq:psi_pi}$, $\eqref{eq:Sigma_cp_n=even}$ and $\eqref{eq:p_pi}$ for the first, third and last equalities,
respectively.
Noting that $\#\Sigma_\odd (n/d)=\floor{n/2d}$ and
$p^-_{\floor{n/2d}}(x)=\pi^{\Sigma_\odd (n/d)}_{\floor{n/2d}}(n/d,x)=\pi^{d\,\Sigma_\odd (n/d)}_{\floor{n/2d}}(n,x)$,
in order to obtain the second expression $\eqref{eq:psi_p-}$,
it suffices to prove that
\begin{equation}
\label{eq:pi(Sigma)}
\pi^\Sigma_{\#\Sigma}(n,x)=\prod_{i=1}^{\#\Pi_1(n)}\left(\prod_{d\in\Pi_i(n)}
\pi^{d\,\Sigma_\odd (n/d)}_{\floor{n/2d}}(n,x)\right)^{(-1)^{i-1}}.
\end{equation}
$(\mathrm{b}.2)$ \emph{Counting the multiplicity of the factors.}
Suppose $m\in\Sigma$.
We may assume that $m$ is divisible by $p_1\cdots p_i\in\Pi_i(n)$ for some $i$, 
but not by any element of $\Pi_{i+1}(n)$,
so that $m\in p_1\Sigma_\odd (n/p_1)\cap\dots\cap p_i\Sigma_\odd (n/p_i)
\setminus\bigcup_{d\in\Pi_{i+1}(n)}d\,\Sigma_\odd (n/d)$.
Then $(x-2\cos (m\pi /n))$ is included as a factor in $\pi^{d\,\Sigma_\odd (n/d)}_{\floor{n/2d}}(n,x)$
for each divisor $d$ of $p_1\cdots p_i$.
Thus the multiplicity of the factor $(x-2\cos (m\pi /n))$ in the right-hand side of $\eqref{eq:pi(Sigma)}$ is given by
\begin{equation*}
\binom{i}{1}-\binom{i}{2}+\dots +(-1)^{i}\binom{i}{i}=1-(1-1)^i=1,
\end{equation*}
so that the right-hand side of $\eqref{eq:pi(Sigma)}$ includes $(x-2\cos (m\pi /n))$ as a factor of multiplicity one.
Since the right-hand side of $\eqref{eq:pi(Sigma)}$
does not include $(x-2\cos (m\pi /n))$ for $m\notin\Sigma$ as a factor, both sides of $\eqref{eq:pi(Sigma)}$ are equal.
Hence putting $\eqref{eq:pi(Sigma)}$ into the right-hand side of $\eqref{eq:psi_2n_p-}$
leads to the desired expression $\eqref{eq:psi_p-}$ of $\psi_{2n}(x)$.
This completes the proof of $(\mathrm{i})$.

\noindent$\phantom{\mathrm{i}}(\mathrm{ii})$ Suppose $n=2s+1$ is odd.
We shall follow the same steps as in $(\mathrm{i})$.

\noindent$(\mathrm{a}.1)$ As for the last equality of $\eqref{eq:p+_psi}$, due to Corollary $\ref{cor:factor_pq}$,
it suffices to prove
\begin{equation}\label{eq:p+_psi_2}
\prod_{k=1}^s\left(x-2\cos\frac{2k}{2s+1}\pi\right)
=\prod_{d<n,\, \divides{d}{n}}\psi_{n/d}(x).
\end{equation}
$(\mathrm{a}.2)$ The degree of the right-hand side of $\eqref{eq:p+_psi_2}$ is calculated as
\begin{equation*}
\sum_{d<n,\, \divides{d}{n}}\frac{\phi (n/d)}{2}
=\sum_{\divides{d}{n}}\frac{\phi (d)}{2}-\frac{\phi (1)}{2}=\frac{n-1}{2}=s,
\end{equation*}
which is equal to that of the left-hand side.

Then step $(\mathrm{a}.3)$ is almost the same and $\eqref{eq:p+_psi}$ is proved.

\noindent$(\mathrm{b}.1)$ We rewrite $\psi_n(x)$ as
\begin{gather}
\begin{aligned}\label{eq:psi_n_n=odd}
\psi_n(x)&=\pi^{2\,\Sigma^{1/2}_\cop (n)}_{\phi (n)/2}(n,x)
&=\pi^{\Sigma_\even (n)\setminus\Sigma'}_{\phi (n)/2}(n,x)
=\frac{\pi^{\Sigma_\even (n)}_{\floor{n/2}}(n,x)}{\pi^{\Sigma'}_{\#\Sigma'}(n,x)}
=\frac{p^+_{\floor{n/2}}(x)}{\pi^{\Sigma'}_{\#\Sigma'}(n,x)},
\end{aligned}\\
\text{where}\quad\Sigma'=\bigcup_{p\in\Pi_1(n)}p\,\Sigma_\even (n/p),\notag
\end{gather}
and we used $\eqref{eq:psi_pi}$, $\eqref{eq:Sigma_cp_n=odd}$ and $\eqref{eq:p_pi}$ for the first, second and last equalities,
respectively.

Then step $(\mathrm{b}.2)$ is almost the same and thus
the right-hand side of $\eqref{eq:psi_n_n=odd}$ is expressed by $\eqref{eq:psi_p+}$.
This completes the proof of $(\mathrm{ii})$.

\noindent$(\mathrm{iii})$ Suppose $n=2s=2^j n'$ is even with $k>0$ and odd $n'$.

\noindent$(\mathrm{a}.1)$ As for the last equality of $\eqref{eq:q-_psi}$, due to Corollary $\ref{cor:factor_pq}$,
it suffices to prove
\begin{equation}\label{eq:q-_psi_2}
\prod_{k=1}^s\left(x-2\cos\frac{2k-1}{2s}\pi\right)
=\prod_{\divides{d}{n'}}\psi_{2n/d}(x).
\end{equation}
$(\mathrm{a}.2)$ The degree of the right-hand side of $\eqref{eq:q-_psi_2}$ is calculated as
\begin{equation*}
\sum_{\divides{d}{n'}}\frac{\phi (2n/d)}{2}
=\sum_{\divides{d}{n'}}\frac{\phi (2^{j+1}d)}{2}=\sum_{\divides{d}{n'}}\frac{2^j\phi (d)}{2}=2^{j-1}n'=s,
\end{equation*}
which is equal to that of the left-hand side.

Then step $(\mathrm{a}.3)$ is almost the same and $\eqref{eq:q-_psi}$ is proved.

\noindent$(\mathrm{b}.1)$ In the same way as $\eqref{eq:psi_2n_p-}$, we can rewrite $\psi_{2n}(x)$ as
\begin{equation}\label{eq:psi_2n_n=even}
\psi_{2n}(x)=\frac{q^-_{n/2}(x)}{\pi^{\Sigma''}_{\#\Sigma''}(n,x)},\\
\quad\text{where}\quad\Sigma''=\bigcup_{p\in\Pi_1(n)}p\,\Sigma_\odd (n/p).
\end{equation}

Then step $(\mathrm{b}.2)$ is almost the same and thus
the right-hand side of the first equation of $\eqref{eq:psi_2n_n=even}$ is expressed by $\eqref{eq:psi_q-}$.
This completes the proof of $(\mathrm{iii})$.
\end{proof}
\newpage
\appendix
\section{Table of the minimal polynomials $\psi_n(x)$ of $2\cos (2\pi /n)$ for $n\leqslant 120$}
\renewcommand{\arraystretch}{1.18}
{\centering
\begin{tabular}{r|c}
$n$&$\psi_n$\\
\hline
$1$&$x-2$\\
$2$&$x+2$\\
$3$&$p^+_1$\\
$4$&$q^-_1$\\
$5$&$p^+_2$\\
$6$&$p^-_1$\\
$7$&$p^+_3$\\
$8$&$q^-_2$\\
$9$&$p^+_4/p^+_1$\\
$10$&$p^-_2$\\
$11$&$p^+_5$\\
$12$&$q^-_3/q^-_1$\\
$13$&$p^+_6$\\
$14$&$p^-_3$\\
$15$&$p^+_7/(p^+_2p^+_1)$\\
$16$&$q^-_4$\\
$17$&$p^+_8$\\
$18$&$p^-_4/p^-_1$\\
$19$&$p^+_9$\\
$20$&$q^-_5/q^-_1$\\
$21$&$p^+_{10}/(p^+_3p^+_1)$\\
$22$&$p^-_5$\\
$23$&$p^+_{11}$\\
$24$&$q^-_6/q^-_2$\\
$25$&$p^+_{12}/p^+_2$\\
$26$&$p^-_6$\\
$27$&$p^+_{13}/p^+_4$\\
$28$&$q^-_7/q^-_1$\\
$29$&$p^+_{14}$\\
$30$&$p^-_7/(p^-_2p^-_1)$\\
$31$&$p^+_{15}$\\
$32$&$q^-_8$\\
$33$&$p^+_{16}/(p^+_5p^+_1)$\\
$34$&$p^-_8$\\
$35$&$p^+_{17}/(p^+_3p^+_2)$\\
$36$&$q^-_9/q^-_3$\\
$37$&$p^+_{18}$\\
$38$&$p^-_9$\\
$39$&$p^+_{19}/(p^+_6p^+_1)$\\
$40$&$q^-_{10}/q^-_2$\\
\end{tabular}
\hspace{0.08\hsize}
\begin{tabular}{r|c}
$n$&$\psi_n$\\
\hline
$41$&$p^+_{20}$\\
$42$&$p^-_{10}/(p^-_3p^-_1)$\\
$43$&$p^+_{21}$\\
$44$&$q^-_{11}/q^-_1$\\
$45$&$p^+_{22}\, p^+_1/(p^+_7p^+_4)$\\
$46$&$p^-_{11}$\\
$47$&$p^+_{23}$\\
$48$&$q^-_{12}/q^-_4$\\
$49$&$p^+_{24}/p^+_3$\\
$50$&$p^-_{12}/p^-_2$\\
$51$&$p^+_{25}/(p^+_8p^+_1)$\\
$52$&$q^-_{13}/q^-_1$\\
$53$&$p^+_{26}$\\
$54$&$p^-_{13}/p^-_4$\\
$55$&$p^+_{27}/(p^+_5p^+_2)$\\
$56$&$q^-_{14}/q^-_2$\\
$57$&$p^+_{28}/(p^+_9p^+_1)$\\
$58$&$p^-_{14}$\\
$59$&$p^+_{29}$\\
$60$&$q^-_{15}\, q^-_1/(q^-_5q^-_3)$\\
$61$&$p^+_{30}$\\
$62$&$p^-_{15}$\\
$63$&$p^+_{31}\, p^+_1/(p^+_{10}\, p^+_4)$\\
$64$&$q^-_{16}$\\
$65$&$p^+_{32}/(p^+_6p^+_2)$\\
$66$&$p^-_{16}/(p^-_5p^-_1)$\\
$67$&$p^+_{33}$\\
$68$&$q^-_{17}/q^-_1$\\
$69$&$p^+_{34}/(p^+_{11}\, p^+_1)$\\
$70$&$p^-_{17}/(p^-_3p^-_2)$\\
$71$&$p^+_{35}$\\
$72$&$q^-_{18}/q^-_6$\\
$73$&$p^+_{36}$\\
$74$&$p^-_{18}$\\
$75$&$p^+_{37}\, p^+_2/(p^+_{12}\, p^+_7)$\\
$76$&$q^-_{19}/q^-_1$\\
$77$&$p^+_{38}/(p^+_5p^+_3)$\\
$78$&$p^-_{19}/(p^-_6p^-_1)$\\
$79$&$p^+_{39}$\\
$80$&$q^-_{20}/q^-_4$
\end{tabular}
\hspace{0.08\hsize}
\begin{tabular}{r|c}
$n$&$\psi_n$\\
\hline
$81$&$p^+_{40}/p^+_{13}$\\
$82$&$p^-_{20}$\\
$83$&$p^+_{41}$\\
$84$&$q^-_{21}\, q^-_1/(q^-_7q^-_3)$\\
$85$&$p^+_{42}/(p^+_8p^+_2)$\\
$86$&$p^-_{21}$\\
$87$&$p^+_{43}/(p^+_{14}\, p^+_1)$\\
$88$&$q^-_{22}/q^-_2$\\
$89$&$p^+_{44}$\\
$90$&$p^-_{22}\, p^-_1/(p^-_7p^-_4)$\\
$91$&$p^+_{45}/(p^+_6p^+_3)$\\
$92$&$q^-_{23}/q^-_1$\\
$93$&$p^+_{46}/(p^+_{15}\, p^+_1)$\\
$94$&$p^-_{23}$\\
$95$&$p^+_{47}/(p^+_9p^+_2)$\\
$96$&$q^-_{24}/q^-_8$\\
$97$&$p^+_{48}$\\
$98$&$p^-_{24}/p^-_3$\\
$99$&$p^+_{49}\, p^+_1/(p^+_{16}\, p^+_4)$\\
$100$&$q^-_{25}/q^-_5$\\
$101$&$p^+_{50}$\\
$102$&$p^-_{25}/(p^-_8p^-_1)$\\
$103$&$p^+_{51}$\\
$104$&$q^-_{26}/q^-_2$\\
$105$&$p^+_{52}\, p^+_3p^+_2p^+_1/(p^+_{17}\, p^+_{10}\, p^+_7)$\\
$106$&$p^-_{26}$\\
$107$&$p^+_{53}$\\
$108$&$q^-_{27}/q^-_9$\\
$109$&$p^+_{54}$\\
$110$&$p^-_{27}/(p^-_5p^-_2)$\\
$111$&$p^+_{55}/(p^+_{18}\, p^+_1)$\\
$112$&$q^-_{28}/q^-_4$\\
$113$&$p^+_{56}$\\
$114$&$p^-_{28}/(p^-_9p^-_1)$\\
$115$&$p^+_{57}/(p^+_{11}\, p^+_2)$\\
$116$&$q^-_{29}/q^-_1$\\
$117$&$p^+_{58}\, p^+_1/(p^+_{19}\, p^+_4)$\\
$118$&$p^-_{29}$\\
$119$&$p^+_{59}/(p^+_8p^+_3)$\\
$120$&$q^-_{30}\, q^-_2/(q^-_{10}\, q^-_6)$\\
\end{tabular}
}

\end{document}